\documentclass[11pt]{amsart}
\usepackage[letterpaper,margin=1.12in]{geometry}
\usepackage[T1]{fontenc}
\usepackage[utf8]{inputenc}
\usepackage{amsmath,amssymb,amsthm,mathtools}
\usepackage{enumitem}
\usepackage{booktabs}
\usepackage{array}
\usepackage{microtype}
\usepackage[colorlinks=true,linkcolor=blue,citecolor=blue,urlcolor=blue]{hyperref}
\usepackage[capitalise,noabbrev]{cleveref}
\numberwithin{equation}{section}
\setlist{leftmargin=2em}

\newtheorem{theorem}{Theorem}[section]
\newtheorem{proposition}[theorem]{Proposition}
\newtheorem{corollary}[theorem]{Corollary}
\newtheorem{lemma}[theorem]{Lemma}
\newtheorem{definition}[theorem]{Definition}

\theoremstyle{remark}
\newtheorem{remark}[theorem]{Remark}

\newtheorem{question}[theorem]{Question}

\crefname{theorem}{Theorem}{Theorems}
\Crefname{theorem}{Theorem}{Theorems}
\crefname{proposition}{Proposition}{Propositions}
\Crefname{proposition}{Proposition}{Propositions}
\crefname{corollary}{Corollary}{Corollaries}
\Crefname{corollary}{Corollary}{Corollaries}
\crefname{lemma}{Lemma}{Lemmas}
\Crefname{lemma}{Lemma}{Lemmas}
\crefname{definition}{Definition}{Definitions}
\Crefname{definition}{Definition}{Definitions}
\crefname{remark}{Remark}{Remarks}
\Crefname{remark}{Remark}{Remarks}
\crefname{example}{Example}{Examples}
\Crefname{example}{Example}{Examples}
\crefname{question}{Question}{Questions}
\Crefname{question}{Question}{Questions}
\crefname{section}{Section}{Sections}
\Crefname{section}{Section}{Sections}
\crefname{equation}{Equation}{Equations}
\Crefname{equation}{Equation}{Equations}
\crefformat{theorem}{Theorem~#2#1#3}
\Crefformat{theorem}{Theorem~#2#1#3}
\crefmultiformat{theorem}{Theorems~#2#1#3}{ and~#2#1#3}{, #2#1#3}{, and~#2#1#3}
\Crefmultiformat{theorem}{Theorems~#2#1#3}{ and~#2#1#3}{, #2#1#3}{, and~#2#1#3}
\crefformat{lemma}{Lemma~#2#1#3}
\Crefformat{lemma}{Lemma~#2#1#3}
\crefmultiformat{lemma}{Lemmas~#2#1#3}{ and~#2#1#3}{, #2#1#3}{, and~#2#1#3}
\Crefmultiformat{lemma}{Lemmas~#2#1#3}{ and~#2#1#3}{, #2#1#3}{, and~#2#1#3}
\crefformat{corollary}{Corollary~#2#1#3}
\Crefformat{corollary}{Corollary~#2#1#3}
\crefmultiformat{corollary}{Corollaries~#2#1#3}{ and~#2#1#3}{, #2#1#3}{, and~#2#1#3}
\Crefmultiformat{corollary}{Corollaries~#2#1#3}{ and~#2#1#3}{, #2#1#3}{, and~#2#1#3}
\crefformat{proposition}{Proposition~#2#1#3}
\Crefformat{proposition}{Proposition~#2#1#3}
\crefmultiformat{proposition}{Propositions~#2#1#3}{ and~#2#1#3}{, #2#1#3}{, and~#2#1#3}
\Crefmultiformat{proposition}{Propositions~#2#1#3}{ and~#2#1#3}{, #2#1#3}{, and~#2#1#3}
\crefformat{definition}{Definition~#2#1#3}
\Crefformat{definition}{Definition~#2#1#3}
\crefformat{remark}{Remark~#2#1#3}
\Crefformat{remark}{Remark~#2#1#3}
\crefformat{question}{Question~#2#1#3}
\Crefformat{question}{Question~#2#1#3}
\crefformat{section}{Section~#2#1#3}
\Crefformat{section}{Section~#2#1#3}
\crefmultiformat{section}{Sections~#2#1#3}{ and~#2#1#3}{, #2#1#3}{, and~#2#1#3}
\Crefmultiformat{section}{Sections~#2#1#3}{ and~#2#1#3}{, #2#1#3}{, and~#2#1#3}

\newcommand{\QQ}{\mathbb Q}
\newcommand{\ZZ}{\mathbb Z}
\newcommand{\QM}{\widetilde M}

\newcommand{\Lam}{\Lambda^{\ast}}
\newcommand{\qbr}[1]{\left\langle #1\right\rangle_q}
\newcommand{\wt}{\operatorname{wt}}
\newcommand{\im}{\operatorname{im}}
\newcommand{\Span}{\operatorname{span}}
\newcommand{\End}{\operatorname{End}}
\newcommand{\Hom}{\operatorname{Hom}}
\newcommand{\rank}{\operatorname{rank}}
\newcommand{\Id}{\operatorname{id}}

\newcommand{\Top}{\operatorname{Top}}

\title[Hecke lifts for partition q-brackets]{Obstructions and kernel transport for Hecke lifts of partition q-brackets}
\author{Levi Segal}
\address{Stanford Online High School, 415 Broadway Academy Hall, Floor 2, 8853, 415 Broadway, Redwood City, CA 94063, USA}
\email{levisegal0@gmail.com}
\date{}


\begin{document}
\begin{abstract}
By the Bloch--Okounkov theorem, the \(q\)-bracket sends shifted
symmetric functions on partitions to quasimodular forms. Following a
question of van Ittersum, we study whether the Hecke action on
quasimodular forms lifts through this map, in the sense of exact lifts
\(A_{n,2k}\) with \(\qbr{A_{n,2k}f}=T_{n,2k}\qbr{f}\). Our main
theorem is that Zagier's lowering operator \(B=\frac12(D-\partial^2)\)
is injective on the genuine homogeneous subspace in every weight at
least \(4\), although it is not injective on the formal algebra. For
exact lifts compatible with lowering, this transports Hecke actions on
\(q\)-bracket kernels between adjacent weights, with divisibility
consequences for characteristic polynomials and \(p\)-adic constraints
on kernel eigenvalues. In fixed weight, we classify all exact lifts by
an action on the kernel and a kernel-valued cocycle relative to a
section. Two obstruction results show that no exact lift is an algebra
homomorphism, and that even strict compatibility with multiplication by
\(Q_2\) already fails in weight \(6\). Finally, we construct exact
lifts with scalar kernel action whenever the \(q\)-bracket image is
Hecke-stable, and exact rational computations give \(q\)-bracket
surjectivity through weight \(16\) and therefore existence of Hecke lift through weight \(16\).
\end{abstract}
\maketitle
\section{Introduction and main results}
\label{sec:intro}

Let \(\mathcal P\) denote the set of all partitions. If
\(f:\mathcal P\to\QQ\) is a function on partitions, its normalized
\(q\)-bracket is
\begin{equation}
\label{eq:qbracket}
        \qbr{f}
        =
        \frac{\sum_{\lambda\in\mathcal P} f(\lambda)q^{|\lambda|}}
             {\sum_{\lambda\in\mathcal P} q^{|\lambda|}}.
\end{equation}
Thus \(\qbr{f}\) is the generating series of the statistic \(f\),
normalized so that \(\qbr{1}=1\). If \(f\) has at most polynomial growth
in \(|\lambda|\), as will be the case for all functions considered here,
the series converges for \(|q|<1\), and setting \(q=e^{2\pi i\tau}\) one
may regard the \(q\)-bracket as a holomorphic function of
\(\tau\in\mathbb H\), the upper-half part of the complex plane.

There are some very simple, but nontrivial examples, producing modular objects through this map.
Consider the statistic \(Q_2(\lambda)=|\lambda|-\frac1{24}\). Since
\(\sum_\lambda|\lambda|\,q^{|\lambda|}=q\frac{d}{dq}\sum_\lambda
q^{|\lambda|}\), logarithmic differentiation of the partition generating
function gives
\begin{equation}
\label{eq:Q2-bracket}
        \qbr{Q_2}
        =
        \sum_{m\geq 1}\sigma_1(m)q^m-\frac1{24}
        =
        -\frac{E_2}{24},
\end{equation}
where \(\sigma_r(n)=\sum_{d\mid n}d^r\) and
\(E_2=1-24\sum_{m\geq 1}\sigma_1(m)q^m\) is the quasimodular Eisenstein
series of weight \(2\).

Some modular behavior is to be expected though. The denominator of
\eqref{eq:qbracket} is the usual partition generating function,
\[
        \sum_{\lambda\in\mathcal P}q^{|\lambda|}
        =
        \prod_{n\geq 1}(1-q^n)^{-1}
        =
        q^{1/24}\eta(\tau)^{-1},
\]
where \(\eta\) is the Dedekind eta-function, so that
\(\eta(\tau)^{-1}\qbr{f}=\sum_\lambda f(\lambda)q^{|\lambda|-1/24}\).
For simple statistics such as polynomials in \(|\lambda|\),
quasimodularity therefore follows from the modular behavior of \(\eta\)
together with the fact that the ring of quasimodular forms is closed
under \(q\frac{d}{dq}\), exactly as in \eqref{eq:Q2-bracket}. The
theorem of Bloch and Okounkov, proved in the course of computing the
character of the infinite wedge representation \cite{BlochOkounkov},
goes far beyond such cases. Zagier's shifted generators \(Q_j\), whose precise normalization
is recalled in \cref{sec:prelim} (the function \(Q_2\) above is one of
them, and \(Q_1=0\) identically on partitions), generate the algebra
\[
        \Lam=\QQ[Q_2,Q_3,Q_4,\ldots]
\]
of shifted symmetric functions, graded by \(\wt(Q_j)=j\). In Zagier's
normalization \cite{ZagierBO}, originally proven by Bloch and Okounkov\cite{BlochOkounkov}, the \(q\)-bracket of every homogeneous
shifted symmetric function of weight \(m\) is a quasimodular form of
weight \(m\) for \(\mathrm{SL}_2(\ZZ)\):
\[
        \qbr{\Lam_m}\subseteq\QM_m,
        \qquad
        \QM=\QQ[E_2,E_4,E_6].
\]
There is no a priori reason for the generating series of a partition
statistic to satisfy any transformation law. A single identity like
\eqref{eq:Q2-bracket} can be checked by hand; the content of the
theorem is that an entire polynomial algebra of partition statistics
lands inside \(\QM\), with the grading preserved.
The bracket has since been connected to multiple zeta values
\cite{BachmannvanIttersum}, to Eichler integrals of Eisenstein series
\cite{BringmannOnoWagner}, and to \(p\)-adic modular forms
\cite{GriffinJamesonTrebatLeder}. Van Ittersum characterized when the
bracket of a shifted symmetric function is modular rather than merely
quasimodular \cite{vanIttersumModular}, and proved symmetric and
congruence-subgroup versions of the theorem
\cite{vanIttersumSymmetric,vanIttersumCongruence}.

Following Zagier~\cite{ZagierBO}, it is convenient to keep \(Q_1\) as a formal
generator. Set
\[
        R=\QQ[Q_1,Q_2,Q_3,\ldots],
\]
graded in the same way, write \(R_m\) for the homogeneous part of weight
\(m\), and
\[
        L_m=\Lam\cap R_m
\]
for the genuine homogeneous subspace. The \(q\)-bracket extends to
\(R\), and every monomial containing \(Q_1\) has bracket zero, since
\(Q_1\) vanishes on partitions. In even weight, let
\[
        q_{2k}:R_{2k}\longrightarrow\QM_{2k},
        \qquad
        q_{2k}(f)=\qbr{f},
\]
and set
\[
        I_{2k}=\im(q_{2k}),
        \qquad
        K_{2k}=\ker q_{2k},
        \qquad
        K^\Lambda_{2k}=K_{2k}\cap L_{2k};
\]
odd-weight brackets vanish outright, since \(\QM\) has no odd-weight
part. The kernel is not an artifact of the formal enlargement: genuine
relations invisible to the \(q\)-bracket appear from weight \(6\) on.
The first is
\begin{equation}
\label{eq:g6-def}
        g_6=161Q_6-109Q_4Q_2+77Q_3^2+5Q_2^3,
        \qquad
        \qbr{g_6}=0,
\end{equation}
and their number grows quickly: in weights \(4,6,\ldots,16\) the genuine
kernels \(K^\Lambda_{2k}\) have dimensions \(0,1,3,7,14,26,45\)
(\cref{sec:computations}). In fact the kernel eventually dominates:
\(\dim L_{2k}=p(2k)-p(2k-1)\), where \(p\) is the partition function,
grows faster than any power of \(k\) by the Hardy--Ramanujan asymptotic
\cite{HardyRamanujan}, while \(\dim\QM_{2k}\) grows only
quadratically. In large weight, almost every shifted symmetric function
lies in the kernel of the \(q\)-bracket.

Quasimodular forms carry Hecke operators. Hecke operators are among the
basic structures in the theory of modular forms: analytically, they are
averaging operators over finite sets of lattice transformations;
geometrically, they come from isogeny correspondences between elliptic
curves. At level one, in the normalization recalled in
\cref{sec:prelim}, they satisfy the composition law
\begin{equation}
\label{eq:hecke-comp}
        T_{m,w}T_{n,w}
        =
        \sum_{d\mid(m,n)}d^{w-1}T_{mn/d^2,w}
\end{equation}
and act on the Eisenstein series by \(T_{n,w}E_w=\sigma_{w-1}(n)E_w\)
for \(w\geq 4\). Their simultaneous eigenvectors organize spaces of
modular forms, and their eigenvalues, which for a normalized
eigenform are its Fourier coefficients, carry the arithmetic. On quasimodular
forms the operators may be defined by the same averaging, or through the
almost-holomorphic lifts of Kaneko and Zagier
\cite{KanekoZagier,Zagier123}; for the Hecke action on quasimodular
forms see also Movasati \cite{Movasati}.

This paper concerns a question posed by van Ittersum on the problem
list of the Cologne conference on modular forms and \(q\)-series
\cite{vanIttersumOpen}: can the Hecke action on quasimodular forms be
realized on shifted symmetric functions through the \(q\)-bracket, or
can one prove that natural lifts do not exist?

\begin{definition}[Exact lifts]
\label{def:exact-lift}
An \emph{exact lift} of \(T_{n,2k}\) is a linear map
\(A_{n,2k}:R_{2k}\to R_{2k}\) such that
\begin{equation}
\label{eq:exact-lift}
        q_{2k}A_{n,2k}=T_{n,2k}q_{2k}.
\end{equation}
A family \(\{A_{n,2k}\}_{n\geq 1}\) of exact lifts is a \emph{full
exact Hecke lift} in weight \(2k\) if it satisfies the relations
\eqref{eq:hecke-comp}. The same terms are used verbatim for operators
on \(\Lam\). When \(I_{2k}\) is stable under \(T_{n,2k}\), we write
\(T^{\mathrm{quot}}_{n,2k}=T_{n,2k}|_{I_{2k}}\) for the induced Hecke
operator on the \(q\)-bracket image; equivalently, this is the Hecke
action on the quotient \(R_{2k}/K_{2k}\cong I_{2k}\).
\end{definition}

An
exact lift would place the Hecke action on the partition side of the
bracket: the operator would act on the statistic \(f\) itself, before
its generating series is formed. Hecke eigenforms, Eisenstein
components, and the Hecke relations could then be approached through
shifted symmetric functions and their relations, rather than only
through Fourier coefficients or polynomial expressions in
\(E_2,E_4,E_6\); conversely, the combinatorial structure of \(\Lam\) (its
generators, its grading, evaluation at particular partitions) would
bear directly on Hecke theory. In this sense, the lifting
problem asks whether the arithmetic symmetries of quasimodular forms
have a combinatorial realization before passing to generating series.
The question allows answers in both directions, either an explicit
construction or a proof that some class of lifts cannot exist, and the
results below contain both. The difficulty is the kernel. Since \(q_{2k}\) is not injective, the
action downstairs determines a lift only up to its behavior on
\(K_{2k}\) and up to correction terms with values in \(K_{2k}\); by
the dimension counts above, this undetermined part is eventually most
of the space.

One basic operator on quasimodular forms does admit a natural lift,
and much of this paper is built around it. Let
\(d:\QM_w\to\QM_{w-2}\) be the derivation determined by
\[
        d(E_2)=12,\qquad d(E_4)=d(E_6)=0,
\]
the lowering operator, which up to normalization is differentiation
with respect to \(E_2\). Zagier \cite{ZagierBO} constructed an operator
doing the same job upstairs: let \(\partial\) be the derivation of
\(R\) with \(\partial(Q_j)=Q_{j-1}\) (so \(\partial(Q_1)=Q_0=1\)), and
set
\begin{equation}
\label{eq:B-def}
        D=
        \sum_{a,b\geq 0}
        \binom{a+b}{a}Q_{a+b}
        \frac{\partial^2}{\partial Q_{a+1}\partial Q_{b+1}},
        \qquad
        B=\frac12(D-\partial^2).
\end{equation}
Then \(B:R_m\to R_{m-2}\), and Zagier's identity states that
\begin{equation}
\label{eq:zagier-lower}
        d\qbr{f}=\qbr{Bf}
        \qquad (f\in R).
\end{equation}
Thus the lifting question has an affirmative answer for at least one
natural operator on \(\QM\), provided one works in the formal
algebra. This is where the formal generator \(Q_1\) enters. The
operator \(B\) does not preserve the genuine algebra: for the kernel
element \(g_6\) of \eqref{eq:g6-def} one computes
\begin{equation}
\label{eq:Bg6}
        B(g_6)=32Q_3Q_1-15Q_2Q_1^2,
\end{equation}
in which every monomial contains \(Q_1\). If one insists on staying
inside genuine shifted symmetric functions and sets \(Q_1=0\) after
applying \(B\), the resulting operator therefore kills \(g_6\) and is
not injective. On the formal algebra, \(B\) is not injective either: a
direct check gives \(B(Q_1^2)=0\). The main theorem of this paper is
that between these two failures there is a true statement: \(B\) is
injective on the genuine subspace, provided its image is allowed to be
formal.

\begin{theorem}[Injectivity of \(B\); see \cref{thm:Binj,cor:Binj-kernel}]
\label{thm:main-injectivity}
For every \(m\geq 4\), the restriction
\[
        B:L_m\longrightarrow R_{m-2}
\]
is injective. In particular, \(B\) embeds the genuine \(q\)-bracket
kernel \(K^\Lambda_{2k}\) into the formal kernel \(K_{2k-2}\) for every
\(k\geq 2\).
\end{theorem}

Concretely, no genuine relation among the shifted generators dies
under lowering, so the \(q\)-bracket kernels in adjacent weights are
linked. The proof is self-contained. Passing to the
associated graded algebra for the filtration of \(R\) by monomial length reduces the
statement to the injectivity of the square of a first-order
differential operator on symmetric polynomials divisible by
\(x_1\cdots x_r\), which is checked directly. For this paper, what matters is the effect on kernels. Applying \(d\) to the Hecke operators gives
the twisted commutation rule
\begin{equation}
\label{eq:dT-twist}
        d(T_{n,2k}F)=nT_{n,2k-2}(dF)
        \qquad (F\in\QM_{2k}),
\end{equation}
so for exact lifts in adjacent weights the combination
\(BA_{n,2k}-nA_{n,2k-2}B\) takes values in the kernel \(K_{2k-2}\)
(\cref{lem:twisted-lowering}). Call a family of exact lifts
\emph{kernel-strict} if the lifts preserve the genuine kernels and this
combination vanishes there, that is,
\(BA_{n,2k}u=nA_{n,2k-2}Bu\) for all \(u\in K^\Lambda_{2k}\)
(\cref{def:kernel-strict}). For kernel-strict lifts,
\cref{thm:main-injectivity} transports structure to the lower weight.
The subspace \(W_{2k-2}=B(K^\Lambda_{2k})\) of \(K_{2k-2}\) is
Hecke-stable, the action of \(A_{n,2k}\) on \(K^\Lambda_{2k}\) is
conjugate under \(B\) to \(n\) times the action of \(A_{n,2k-2}\) on
\(W_{2k-2}\), and after rescaling the variable the characteristic
polynomial upstairs divides the characteristic polynomial on
\(K_{2k-2}\) (\cref{thm:kernel-transport,thm:spectral-divisibility}).
If in addition the action downstairs preserves a \(\ZZ_p\)-lattice,
then every eigenvalue of \(A_{p,2k}\) on \(K^\Lambda_{2k}\) is
divisible by \(p\); in particular, the genuine kernel action of a
kernel-strict lift can never be the Eisenstein scalar \(1+p^{2k-1}\)
(\cref{cor:padic}).

The remaining results work in a fixed weight and describe the possible
exact lifts there. Suppose
\(I_{2k}\) is stable under \(T_{n,2k}\); by the computations of
\cref{sec:computations} this holds, with \(I_{2k}=\QM_{2k}\), in every
even weight up to \(16\). Choosing a linear section
\(s_{2k}:I_{2k}\to R_{2k}\) of \(q_{2k}\) splits
\(R_{2k}=s_{2k}(I_{2k})\oplus K_{2k}\).

\begin{theorem}[Fixed-weight structure; see
\cref{thm:block-classification,thm:cocycle}]
\label{thm:main-structure}
Relative to the splitting \(R_{2k}=s_{2k}(I_{2k})\oplus K_{2k}\), the
exact lifts of \(T_{n,2k}\) are exactly the operators
\[
        A_{n,2k}
        =
        \begin{pmatrix}
        T^{\mathrm{quot}}_{n,2k} & 0\\
        C_{n,2k} & N_{n,2k}
        \end{pmatrix},
        \qquad
        N_{n,2k}\in\End(K_{2k}),
        \quad
        C_{n,2k}\in\Hom(I_{2k},K_{2k}).
\]
A family of such operators satisfies the Hecke relations
\eqref{eq:hecke-comp} if and only if the kernel operators \(N_{n,2k}\)
satisfy \eqref{eq:hecke-comp} on \(K_{2k}\) and the maps \(C_{n,2k}\)
satisfy
\begin{equation}
\label{eq:cocycle}
        C_{m,2k}T^{\mathrm{quot}}_{n,2k}
        +
        N_{m,2k}C_{n,2k}
        =
        \sum_{d\mid(m,n)}d^{2k-1}C_{mn/d^2,2k}.
\end{equation}
\end{theorem}

Equation \eqref{eq:cocycle} is a cocycle identity in the usual sense:
replacing the section \(s_{2k}\) by \(s_{2k}+h\) replaces \(C_{n,2k}\)
by \(C_{n,2k}+N_{n,2k}h-hT^{\mathrm{quot}}_{n,2k}\), a coboundary. A
full exact lift in weight \(2k\) is therefore the same thing as a Hecke
module structure on \(K_{2k}\) together with a short exact sequence of sequence of hecke modules
\[
        0\longrightarrow K_{2k}\longrightarrow R_{2k}
        \longrightarrow I_{2k}\longrightarrow 0.
\]
The proofs are direct block-matrix computations. The value of the
statement is that it locates the free choices of a lift, namely the
kernel action \(N_{n,2k}\) and the correction term \(C_{n,2k}\).

The obstruction results show that these correction terms are
unavoidable; they are the negative half promised above. We test two
conditions. The first asks the lift to be a ring map. The second asks
only that it commute with multiplication by the single generator
\(Q_2\), up to the weight-\(2\) Hecke eigenvalue \(\sigma_1(p)=1+p\).
Both fail, and the second fails already in weight \(6\). These questions are not assumed to be natural, rather tests and examples of the rigidity of potential hecke lifts.

\begin{theorem}[Obstructions; see
\cref{thm:no-multiplicative,thm:no-q2-semi}]
\label{thm:main-obstructions}
Let \(p\) be prime.
\begin{enumerate}[label=\textup{(\roman*)}]
\item No graded algebra endomorphism \(A_p:\Lam\to\Lam\) satisfies
\(\qbr{A_pf}=T_{p,\wt(f)}\qbr{f}\) for every homogeneous
\(f\in\Lam\).
\item Let \(e_4=Q_2^2+2Q_4\), so that \(\qbr{e_4}=\frac1{240}E_4\).
Then no exact lift on \(\Lam\) in weights \(4\) and \(6\) satisfies
\[
        A_{p,6}(Q_2e_4)=(1+p)Q_2A_{p,4}(e_4).
\]
\end{enumerate}
\end{theorem}

Part (i) is not surprising, since Hecke operators are not ring
homomorphisms; its proof is a short comparison of eigenvalues on
\(\qbr{Q_2}\) and \(\qbr{Q_2^2}\). Part (ii) is the stronger
statement. Even semilinearity with respect to a single multiplication, tested at
a single element on the Eisenstein line, is incompatible with
exactness, and part (ii) implies part (i), since a multiplicative lift
is in particular strictly \(Q_2\)-semilinear on \(e_4\). Every exact
lift therefore carries
nonzero correction terms of the type classified by
\cref{thm:main-structure}, and the defect
\(A_{p,6}(Q_2e_4)-(1+p)Q_2A_{p,4}(e_4)\) is a concrete nonzero element
of \(L_6\) attached to each one.

Despite the obstructions, exact lifts exist whenever \(I_{2k}\) is
Hecke-stable, or if the q-bracket is surjective, and simple ones can be written down. Fix a section
\(s_{2k}\) as above. For \(2k\geq 4\), define the
\emph{scalar-kernel lift} (see \cref{def:transported-scalar})
\begin{equation}
\label{eq:sc-def}
        A^{\mathrm{sc}}_{n,2k}\bigl(s_{2k}(\phi)+u\bigr)
        =
        s_{2k}(T_{n,2k}\phi)+n^{k-2}\sigma_3(n)u,
        \qquad
        \phi\in I_{2k},\ u\in K_{2k}.
\end{equation}
The scalars \(\lambda_{n,2k}=n^{k-2}\sigma_3(n)\) are chosen to make
this work. By the classical identity
\(\sigma_3(m)\sigma_3(n)=\sum_{d\mid(m,n)}d^3\sigma_3(mn/d^2)\) they
satisfy the relations \eqref{eq:hecke-comp} in weight \(2k\), and they
obey \(\lambda_{n,2k}=n\lambda_{n,2k-2}\), matching the twist in
\eqref{eq:dT-twist}.

\begin{corollary}[Scalar-kernel lifts and weights at most \(16\); see
\cref{thm:scalar-lift,cor:scalar-strict,prop:kerneldims}]
\label{cor:main-16}
Let \(2k\geq 4\) and suppose \(I_{2k}\) is stable under the level-one
Hecke action. Then \(\{A^{\mathrm{sc}}_{n,2k}\}_{n\geq 1}\) is a full
exact Hecke lift, and if the construction is made in both weights
\(2k\) and \(2k-2\), with \(k\geq 3\), then
\[
        BA^{\mathrm{sc}}_{n,2k}u=nA^{\mathrm{sc}}_{n,2k-2}Bu
        \qquad (u\in K_{2k}).
\]
Moreover, exact rational computation gives
\[
        \qbr{L_{2k}}=\QM_{2k}
        \qquad (2k\leq 16),
\]
so in every even weight at most \(16\) the construction produces exact
lifts of the full quasimodular Hecke action, kernel-strict between
adjacent weights.
\end{corollary}

In particular, the hypotheses of the transport results above are not
vacuous, since the scalar-kernel lifts satisfy them. Note also that this is
consistent with the \(p\)-adic constraint above, since the kernel
scalar
\(\lambda_{p,2k}=p^{k-2}(1+p^3)\) is divisible by \(p\) once
\(k\geq 3\). The computation in \cref{cor:main-16} is finite-weight evidence
only; whether \(\qbr{L_{2k}}=\QM_{2k}\) holds in every weight is open,
and is recorded as \Cref{q:surjectivity}. \Cref{app:harmonic} reduces
this surjectivity question to a harmonic statement, conditional on a
depth-triangularity property of multiplication by \(Q_2\).

The paper is organized as follows. \Cref{sec:prelim} fixes the
normalizations for the shifted generators, the Hecke operators, and the
almost-holomorphic model. \Cref{sec:cocycles} proves \cref{thm:main-structure},
\cref{sec:obstructions} proves \cref{thm:main-obstructions}, and
\cref{sec:Binj} proves \cref{thm:main-injectivity} together with its
transport consequences. \Cref{sec:scalar}
constructs the scalar-kernel lifts, \cref{sec:computations} contains
the rank computations through weight \(16\), and \cref{sec:discussion}
collects open questions. \Cref{app:computational} records the
exact-rational rank certificates and the low-weight kernel vectors,
including \eqref{eq:g6-def} and \eqref{eq:Bg6};
\cref{app:harmonic} gives the conditional harmonic reduction.
\section{Preliminaries and normalizations}
\label{sec:prelim}

We retain the notation of \cref{sec:intro}; in particular
\(R=\QQ[Q_1,Q_2,Q_3,\ldots]\), \(\Lam=\QQ[Q_2,Q_3,Q_4,\ldots]\), and
\(L_m=\Lam\cap R_m\). This section fixes the normalizations used in the
proofs.

We use Zagier's normalization of the Bloch--Okounkov shifted generators
\cite{ZagierBO}.

\begin{definition}[Shifted generators]
\label{def:shifted-generators}
For a partition \(\lambda=(\lambda_1,\lambda_2,\ldots)\) and an integer
\(r\geq 0\), set
\[
        P_r(\lambda)
        =
        \sum_{i\geq 1}
        \left[
        \left(\lambda_i-i+\frac12\right)^r
        -
        \left(-i+\frac12\right)^r
        \right].
\]
The sum is finite, since \(\lambda_i=0\) for all sufficiently large
\(i\). The \emph{shifted generators} are
\[
        Q_0=1,\qquad
        Q_j(\lambda)=\frac{P_{j-1}(\lambda)}{(j-1)!}+\beta_j
        \qquad (j>0),
\]
where the constants \(\beta_j\) are determined by
\[
        \frac{z/2}{\sinh(z/2)}
        =
        \sum_{j\geq 0}\beta_jz^j.
\]
\end{definition}

Since \(P_0=0\) and \(P_1(\lambda)=|\lambda|\), while \(\beta_1=0\) and
\(\beta_2=-\frac1{24}\) because the left-hand side is an even function
of \(z\), this normalization produces the two facts already used in
\cref{sec:intro}:
\[
        Q_1=0\quad\text{on partitions},
        \qquad
        Q_2(\lambda)=|\lambda|-\frac1{24}.
\]

For \(f\in R_m\), we write \(q_m(f)=\qbr{f}\). Since \(Q_1\) vanishes on
partitions, \(q_m(R_m)=q_m(L_m)\), and the Bloch--Okounkov theorem in
this normalization gives
\[
        q_m(R_m)\subseteq \QM_m
\]
for every \(m\) \cite{BlochOkounkov,ZagierBO}. The odd-weight pieces of
\(R\) lie in the \(q\)-bracket kernel, since \(\QM\) has no odd-weight
part.
\begin{definition}[Hecke operators]
\label{def:hecke}
The level-one \emph{Hecke operator} of index \(n\) in weight \(w\) acts
on quasimodular forms by
\[
        T_{n,w}F(\tau)
        =
        \frac1n
        \sum_{cc'=n}
        \sum_{0\leq b<c}
        c'^{\,w}
        F\left(\frac{c'\tau+b}{c}\right).
\]
\end{definition}
For fourier-series representations of modular forms, the Hecke operator of prime index \(p\) in
weight \(w\) can also be written as
\begin{equation}
  T_{p,w}\Bigl(\sum_{n\ge0}a_nq^n\Bigr)
  =\sum_{n\ge0}\bigl(a_{np}+p^{w-1}a_{n/p}\bigr)q^n,
\end{equation}
where \(a_{n/p}=0\) unless \(p\mid n\), and the same formula defines
the action on quasimodular forms \cite{Movasati}, although we will not use this form.

With this normalization, the operators satisfy the composition law
\eqref{eq:hecke-comp}, and for Eisenstein series
\[
        T_{n,w}E_w=\sigma_{w-1}(n)E_w
        \qquad (w\geq 4);
\]
see \cite{Zagier123,Movasati} for the Hecke action on quasimodular forms.

Put
\[
        X=(4\pi\operatorname{Im}\tau)^{-1}.
\]
Every \(F\in \QM_w\) has a unique almost-holomorphic lift \(F^*\), a
polynomial in \(X\) with holomorphic coefficients which transforms as a
modular form of weight \(w\); the degree of \(F^*\) in \(X\) is the
\(E_2\)-depth of \(F\) \cite{KanekoZagier,Zagier123}. The coefficient
of the highest power of \(X\) is a fixed nonzero scalar multiple of the
coefficient of the highest power of \(E_2\) in \(F\). Indeed, since
\(E_2^*=E_2-12X\) with this choice of \(X\), and since modular forms have
$E_2$-depth zero, the \(X^k\)-coefficient of \(F^*\) for \(F\in\QM_{2k}\) is
the constant \((-12)^k\) times the \(E_2^k\)-coefficient of \(F\).

Under the substitution \(\tau\mapsto(c'\tau+b)/c\), one has
\begin{equation}
\label{eq:X-transform}
        \operatorname{Im}\left(\frac{c'\tau+b}{c}\right)
        =
        \frac{c'}{c}\operatorname{Im}\tau,
        \qquad
        X\longmapsto \frac{c}{c'}X.
\end{equation}
The transformation rule \eqref{eq:X-transform} is the only property of
the model used in the proof of \cref{lem:top-hecke}; the expansion of
\(F^*\) in powers of \(X\) is needed only once, in the proof of
\cref{lem:twisted-lowering}.

Recall from \cref{sec:intro} the lowering derivation \(d\), determined
by \(d(E_2)=12\) and \(d(E_4)=d(E_6)=0\). In the almost-holomorphic
model, \(d\) is, up to a fixed nonzero normalization, differentiation
in the \(X\)-direction.
Define the algebra homomorphism
\[
        \mu:R\longrightarrow \QQ,
        \qquad
        \mu(Q_j)=\frac{1-j}{j!}.
\]
Zagier's top-depth formula \cite{ZagierBO} says that for
\(f\in R_{2k}\),
\begin{equation}
\label{eq:top-depth}
        \qbr{f}
        =
        -\frac{(2k-3)!!}{(-12)^k}\mu(f)E_2^k
        +
        \text{lower powers of }E_2.
\end{equation}
For \(F\in \QM_{2k}\), let \(\Top(F)\) denote the coefficient of
\(E_2^k\).

\begin{lemma}[Top coefficient under Hecke]
\label{lem:top-hecke}
For \(F\in \QM_{2k}\),
\[
        \Top(T_{n,2k}F)
        =
        n^{k-1}\sigma_1(n)\Top(F).
\]
\end{lemma}

\begin{proof}
Pass to the almost-holomorphic lift \(F^*\) of F. The Hecke average of \(F^*\) and
the lift of \(T_{n,2k}F\) are almost-holomorphic modular forms of
weight \(2k\) that agree at \(X=0\), hence are equal. The \(X^k\)-coefficient of \(F^*\) is the constant
\((-12)^k\Top(F)\), so precomposition with
\(\tau\mapsto(c'\tau+b)/c\) leaves it unchanged, and the substitution
acts on the top term only through the rescaling
\(X\mapsto(c/c')X\) of \eqref{eq:X-transform}. In the summand indexed
by \(cc'=n\), the top \(X^k\)-coefficient is therefore multiplied by
\[
        \frac1n c'^{\,2k}\left(\frac{c}{c'}\right)^k.
\]
There are \(c\) choices of \(b\), so the total multiplier is
\[
        \frac1n\sum_{cc'=n}c^{k+1}c'^{\,k}
        =
        \frac1n\sum_{c\mid n}c^{k+1}\left(\frac nc\right)^k
        =
        n^{k-1}\sigma_1(n).
\]
\end{proof}

\begin{corollary}[Top-term eigenvalue law]
\label{cor:top-eigenvalue}
If \(A_{n,2k}\) is an exact lift, then
\[
        \mu(A_{n,2k}f)
        =
        n^{k-1}\sigma_1(n)\mu(f)
        \qquad (f\in R_{2k}).
\]
\end{corollary}

\begin{proof}
By \eqref{eq:top-depth}, \(\Top(\qbr{f})\) is a fixed nonzero multiple
of \(\mu(f)\). Apply \cref{lem:top-hecke} and exactness
\eqref{eq:exact-lift}.
\end{proof}

Recall from \cref{sec:intro} the derivation \(\partial\), the
second-order operator \(D\), and Zagier's lowering operator
\(B=\frac12(D-\partial^2)\) of \eqref{eq:B-def}, together with Zagier's
identity \eqref{eq:zagier-lower}. In particular, if \(\qbr{f}\) is
modular, then \(\qbr{Bf}=0\); the injectivity theorem for \(B\) in
\cref{sec:Binj} is therefore a statement about the source space
\(L_m\), not about the \(q\)-bracket image.

\begin{lemma}[Twisted lowering]
\label{lem:twisted-lowering}
For \(F\in \QM_{2k}\),
\[
        d(T_{n,2k}F)=nT_{n,2k-2}(dF).
\]
Consequently, if \(A_{n,2k}\) and \(A_{n,2k-2}\) are exact lifts, then
\[
        BA_{n,2k}f-nA_{n,2k-2}Bf\in K_{2k-2}
        \qquad (f\in R_{2k}).
\]
\end{lemma}

\begin{proof}
Write
\[
        F^*(\tau)=\sum_{j=0}^kF_j(\tau)X^j
\]
for the almost-holomorphic lift of \(F\), so that \(d\) is, up to the
fixed normalization above, differentiation with respect to \(X\). In a
Hecke summand indexed by \(cc'=n\), the substitution
\(\tau\mapsto(c'\tau+b)/c\) sends \(X\) to \((c/c')X\) by
\eqref{eq:X-transform}. Applying \(d\) after this substitution
therefore introduces the factor \(c/c'\). The coefficient of the
summand in weight \(2k\) is \(n^{-1}c'^{\,2k}\), so after applying
\(d\) the coefficient becomes
\[
        \frac1n c'^{\,2k}\frac{c}{c'}
        =
        \frac{cc'}n c'^{\,2k-2}
        =
        c'^{\,2k-2}.
\]
The corresponding coefficient in the weight \(2k-2\) Hecke operator is
\(n^{-1}c'^{\,2k-2}\). Hence the differentiated weight-\(2k\) summand
is \(n\) times the corresponding weight-\((2k-2)\) summand applied to
\(dF\). Summing over \(c,c'\), and \(b\), gives
\[
        d(T_{n,2k}F)=nT_{n,2k-2}(dF).
\]

Now let \(f\in R_{2k}\). Using Zagier's identity
\eqref{eq:zagier-lower}, exactness in weight \(2k\), the identity just
proved, and exactness in weight \(2k-2\), we obtain
\[
\begin{aligned}
        \qbr{BA_{n,2k}f}
        &=
        d\qbr{A_{n,2k}f}                         \\
        &=
        dT_{n,2k}\qbr{f}                           \\
        &=
        nT_{n,2k-2}d\qbr{f}                         \\
        &=
        nT_{n,2k-2}\qbr{Bf}                         \\
        &=
        n\qbr{A_{n,2k-2}Bf}.
\end{aligned}
\]
Therefore the \(q\)-bracket of
\[
        BA_{n,2k}f-nA_{n,2k-2}Bf
\]
vanishes, which is exactly the assertion that this difference lies in
\(K_{2k-2}\).
\end{proof}
\section{Kernel-cocycle structure of exact lifts}
\label{sec:cocycles}

Exact lifts have a useful fixed-weight normal form before any naturality condition is imposed.  This section records that normal form and the cocycle equation that the Hecke relations impose on it.  These facts do not solve the lifting problem, but they locate precisely where the correction terms of a lift must live.

Throughout the section, fix an even weight $2k$, assume $I_{2k}$ is stable under $T_{n,2k}$, and choose a linear section
\[
  s_{2k}:I_{2k}\to R_{2k},
  \qquad
  q_{2k}s_{2k}=\Id_{I_{2k}},
\]
so that $R_{2k}=s_{2k}(I_{2k})\oplus K_{2k}$.  As in \cref{def:exact-lift}, $T^{\mathrm{quot}}_{n,2k}=T_{n,2k}|_{I_{2k}}$ denotes the induced Hecke operator on the image.

\begin{theorem}[Fixed-weight classification]
\label{thm:block-classification}
Let $A_{n,2k}:R_{2k}\to R_{2k}$ be an exact lift of $T_{n,2k}$.  Relative to the decomposition
\[
  R_{2k}=s_{2k}(I_{2k})\oplus K_{2k},
\]
there are unique maps
\[
  C_{n,2k}\in\Hom(I_{2k},K_{2k}),
  \qquad
  N_{n,2k}\in\End(K_{2k})
\]
such that
\[
  A_{n,2k}=\begin{pmatrix}
  T^{\mathrm{quot}}_{n,2k}&0\\
  C_{n,2k}&N_{n,2k}
  \end{pmatrix}.
\]
Conversely, every operator of this form is an exact lift.
\end{theorem}

\begin{proof}
The chosen section gives a direct-sum decomposition
\[
  R_{2k}=s_{2k}(I_{2k})\oplus K_{2k}.
\]
Thus every element of $R_{2k}$ is written uniquely as $s_{2k}(\phi)+u$ with $\phi\in I_{2k}$ and $u\in K_{2k}$.  We first determine what exactness implies on the kernel.  If $u\in K_{2k}$, then
\[
  q_{2k}(A_{n,2k}u)=T_{n,2k}q_{2k}(u)=T_{n,2k}(0)=0.
\]
Hence $A_{n,2k}u\in K_{2k}$.  Therefore $A_{n,2k}$ preserves the kernel, so the component from $K_{2k}$ to the quotient $I_{2k}$ is zero.  The restriction of $A_{n,2k}$ to $K_{2k}$ is the uniquely determined map
\[
  N_{n,2k}=A_{n,2k}|_{K_{2k}}\in\End(K_{2k}).
\]

It remains to describe the image of a section element $s_{2k}(\phi)$.  Since $A_{n,2k}$ is exact,
\[
  q_{2k}A_{n,2k}s_{2k}(\phi)=T_{n,2k}q_{2k}s_{2k}(\phi)=T_{n,2k}\phi.
\]
The element $s_{2k}(T_{n,2k}\phi)$ is the chosen representative in $R_{2k}$ of this quotient element.  Therefore the difference
\[
  A_{n,2k}s_{2k}(\phi)-s_{2k}(T_{n,2k}\phi)
\]
has zero q-bracket and lies in $K_{2k}$.  This defines a linear map
\[
  C_{n,2k}:I_{2k}\longrightarrow K_{2k},
  \qquad
  C_{n,2k}(\phi)=A_{n,2k}s_{2k}(\phi)-s_{2k}(T_{n,2k}\phi).
\]
Linearity follows from linearity of $A_{n,2k}$ and of the section.

Combining the two computations, for a general element $s_{2k}(\phi)+u$ we have
\[
  A_{n,2k}(s_{2k}(\phi)+u)
  =s_{2k}(T_{n,2k}\phi)+C_{n,2k}(\phi)+N_{n,2k}(u),
\]
which is exactly the displayed block form.  The maps $C_{n,2k}$ and $N_{n,2k}$ are unique because the decomposition into the section component and the kernel component is unique.

Conversely, suppose an operator has the displayed form.  Then
\[
\begin{aligned}
 q_{2k}A_{n,2k}(s_{2k}(\phi)+u)
 &=q_{2k}\bigl(s_{2k}(T_{n,2k}\phi)+C_{n,2k}(\phi)+N_{n,2k}(u)\bigr) \\
 &=T_{n,2k}\phi,
\end{aligned}
\]
because the last two terms lie in the kernel.  Since $q_{2k}(s_{2k}(\phi)+u)=\phi$, this proves $q_{2k}A_{n,2k}=T_{n,2k}q_{2k}$.
\end{proof}

\begin{remark}
\label{rem:param-count}
The theorem shows that, at fixed weight $2k$, the affine space of exact lifts of $T_{n,2k}$ is modeled on $\Hom(I_{2k},K_{2k})\oplus\End(K_{2k})$.  Writing $m_k=\dim I_{2k}$ and $r_k=\dim K_{2k}$, an exact lift therefore has $m_kr_k+r_k^2=r_k\dim R_{2k}$ free linear parameters.  None of them is constrained until the Hecke relations are imposed (\cref{thm:cocycle}).
\end{remark}

\begin{theorem}[Hecke relations and cocycle data]
\label{thm:cocycle}
Fix the section $s_{2k}$ and, for each $n\ge1$, maps $N_{n,2k}\in\End(K_{2k})$ and $C_{n,2k}\in\Hom(I_{2k},K_{2k})$, and let $A_{n,2k}$ be the exact lift with block form
\[
  A_{n,2k}=\begin{pmatrix}
  T^{\mathrm{quot}}_{n,2k}&0\\
  C_{n,2k}&N_{n,2k}
  \end{pmatrix},
\]
equivalently
\[
  A_{n,2k}\bigl(s_{2k}(\phi)+u\bigr)
  =s_{2k}(T_{n,2k}\phi)+C_{n,2k}(\phi)+N_{n,2k}(u).
\]
Then $\{A_{n,2k}\}_{n\ge1}$ is a full exact Hecke lift if and only if the kernel operators satisfy the Hecke relations
\[
  N_{m,2k}N_{n,2k}=\sum_{d\mid(m,n)}d^{2k-1}N_{mn/d^2,2k}
\]
and the correction maps satisfy the cocycle identity
\[
  C_{m,2k}T^{\mathrm{quot}}_{n,2k}+N_{m,2k}C_{n,2k}
  =\sum_{d\mid(m,n)}d^{2k-1}C_{mn/d^2,2k}.
\]
By \cref{thm:block-classification}, every full exact Hecke lift in weight $2k$ arises in this way.
\end{theorem}

\begin{proof}
Each $A_{n,2k}$ is an exact lift by \cref{thm:block-classification}.  Multiplying two operators in block form gives
\[
  A_{m,2k}A_{n,2k}
  =
  \begin{pmatrix}
  T^{\mathrm{quot}}_{m,2k}T^{\mathrm{quot}}_{n,2k} & 0\\
  C_{m,2k}T^{\mathrm{quot}}_{n,2k}+N_{m,2k}C_{n,2k} & N_{m,2k}N_{n,2k}
  \end{pmatrix},
\]
while
\[
  \sum_{d\mid(m,n)}d^{2k-1}A_{mn/d^2,2k}
  =
  \sum_{d\mid(m,n)}d^{2k-1}
  \begin{pmatrix}
  T^{\mathrm{quot}}_{mn/d^2,2k}&0\\
  C_{mn/d^2,2k}&N_{mn/d^2,2k}
  \end{pmatrix}.
\]
The Hecke relation
\[
  A_{m,2k}A_{n,2k}=\sum_{d\mid(m,n)}d^{2k-1}A_{mn/d^2,2k}
\]
is an equality of block matrices, hence holds if and only if it holds in each block.  The upper-left blocks agree unconditionally: their equality is the composition law \eqref{eq:hecke-comp} on $I_{2k}$, which holds because $I_{2k}$ is Hecke-stable.  Equality of the lower-right blocks is the displayed Hecke relation for the $N_{n,2k}$, and equality of the lower-left blocks is the cocycle identity.  The last assertion is the classification of \cref{thm:block-classification}.
\end{proof}

\begin{corollary}[Prime cases]
\label{cor:prime-cocycle}
Assume $A_{1,2k}=\Id$.  Then $N_{1,2k}=\Id_{K_{2k}}$ and $C_{1,2k}=0$.  If $p\ne q$ are primes, then
\[
  N_{p,2k}N_{q,2k}=N_{pq,2k}
\]
and
\[
  C_{p,2k}T^{\mathrm{quot}}_{q,2k}+N_{p,2k}C_{q,2k}=C_{pq,2k}.
\]
For a prime $p$,
\[
  N_{p,2k}^{2}=N_{p^2,2k}+p^{2k-1}\Id_{K_{2k}},
\]
and
\[
  C_{p,2k}T^{\mathrm{quot}}_{p,2k}+N_{p,2k}C_{p,2k}=C_{p^2,2k}.
\]
\end{corollary}

\begin{proof}
Since $T_{1,2k}=\Id$, the assumption $A_{1,2k}=\Id$ gives $N_{1,2k}=\Id_{K_{2k}}$ and $C_{1,2k}=0$ directly from the block form.  The remaining identities are the relations of \cref{thm:cocycle} specialized to $(m,n)=(p,q)$ and $(m,n)=(p,p)$: in the first case the divisor sum has the single term $d=1$, and in the second it has the terms $d=1$ and $d=p$, the latter contributing $p^{2k-1}N_{1,2k}=p^{2k-1}\Id_{K_{2k}}$ to the kernel relation and $p^{2k-1}C_{1,2k}=0$ to the cocycle relation.
\end{proof}

\begin{remark}[Extension-theoretic reading]
\label{rem:extension}
\Cref{thm:cocycle} has a homological interpretation.  A full exact Hecke lift in weight $2k$ makes
\[
  0\longrightarrow K_{2k}\longrightarrow R_{2k}\xrightarrow{\;q_{2k}\;} I_{2k}\longrightarrow 0
\]
a short exact sequence of Hecke modules: the quotient carries the action $T^{\mathrm{quot}}_{n,2k}$, the kernel carries the action $N_{n,2k}$, and, relative to the chosen splitting, the correction maps $C_{n,2k}$ form a kernel-valued $1$-cocycle.  The cocycle depends on the section only up to coboundary: replacing $s_{2k}$ by $s_{2k}+h$ with $h:I_{2k}\to K_{2k}$ replaces the correction maps by
\[
  C'_{n,2k}
  =A_{n,2k}(s_{2k}+h)-(s_{2k}+h)T^{\mathrm{quot}}_{n,2k}
  =C_{n,2k}+N_{n,2k}h-hT^{\mathrm{quot}}_{n,2k}.
\]
The obstruction results of \cref{sec:obstructions} show that this freedom cannot be removed by naturality conditions: no exact lift is multiplicative, and none is strictly $Q_2$-semilinear.
\end{remark}
\section{Obstructions to natural exact lifts}
\label{sec:obstructions}

The construction in Section~\ref{sec:scalar} gives noncanonical exact lifts.  The results in this section examine two deliberately strong compatibility attempts: multiplicativity and strict compatibility with multiplication by $Q_2$.  Their failure shows where correction terms must enter.

We first record two low-weight facts.  Since
\[
  \mu(Q_2)=-\frac12\neq0,
\]
the q-bracket $\qbr{Q_2}$ is a nonzero multiple of $E_2$.  Define
\begin{equation}
  e_4=Q_2^2+2Q_4.
  \label{eq:e4-def}
\end{equation}
Then
\[
  \mu(e_4)=\mu(Q_2)^2+2\mu(Q_4)=\frac14+2\cdot\left(-\frac18\right)=0,
\]
so $\qbr{e_4}$ has no $E_2^2$ term.  Since
\[
  Q_2(\varnothing)=-\frac1{24},
  \qquad
  Q_4(\varnothing)=\beta_4=\frac7{5760},
\]
we have
\[
  e_4(\varnothing)=\frac1{576}+\frac{14}{5760}=\frac1{240}\neq0.
\]
Thus $\qbr{e_4}$ is a nonzero modular form of weight $4$, hence a nonzero multiple of $E_4$.

\begin{lemma}[Low-weight injectivity on the genuine algebra]
\label{lem:low-weight-inj}
The q-bracket maps
\[
  L_2\to\QM_2,
  \qquad
  L_4\to\QM_4
\]
are injective.  Moreover $Q_2$ maps to a nonzero multiple of $E_2$, and $e_4$ maps to a nonzero multiple of $E_4$.
\end{lemma}

\begin{proof}
The space $L_2$ is one-dimensional, with basis $Q_2$.  We have already observed that $\mu(Q_2)=-1/2\ne0$, so Zagier's top-depth formula gives a nonzero $E_2$ coefficient in $\qbr{Q_2}$.  Hence $Q_2$ does not lie in the q-bracket kernel, and the map $L_2\to\QM_2$ is injective.

Now consider weight $4$.  The space $L_4$ is spanned by $Q_2^2$ and $Q_4$.  Instead of using $Q_4$ directly, use the basis
\[
  Q_2^2,
  \qquad
  e_4=Q_2^2+2Q_4.
\]
The first element has nonzero top $E_2^2$ coefficient, because
\[
  \mu(Q_2^2)=\mu(Q_2)^2=\frac14\ne0.
\]
Thus $\qbr{Q_2^2}$ has a nonzero $E_2^2$ component.  On the other hand, $\mu(e_4)=0$, so $\qbr{e_4}$ has no $E_2^2$ term.  We also computed $e_4(\varnothing)=1/240\ne0$, so the constant term of $\qbr{e_4}$ is nonzero.  Since a weight-$4$ quasimodular form with no $E_2^2$ component is a scalar multiple of $E_4$, this shows that $\qbr{e_4}$ is a nonzero multiple of $E_4$.

Therefore the two images $\qbr{Q_2^2}$ and $\qbr{e_4}$ have different top-depth behavior: one has a nonzero $E_2^2$ term and the other is a nonzero modular form.  They are linearly independent in
\[
  \QM_4=\Span_\QQ(E_2^2,E_4).
\]
Since $\dim L_4=2$, the q-bracket map $L_4\to\QM_4$ is injective.
\end{proof}

\begin{theorem}[No multiplicative exact lift under the strong algebra-homomorphism condition]
\label{thm:no-multiplicative}
For every prime $p$, there is no graded algebra endomorphism
\[
  A_p:\Lam\to\Lam
\]
such that
\[
  \qbr{A_pf}=T_{p,\wt(f)}\qbr{f}
\]
for every homogeneous $f\in\Lam$.
\end{theorem}

\begin{proof}
Assume, for contradiction, that such an algebra endomorphism $A_p$ exists.  In weight $2$, the space $L_2$ is one-dimensional and spanned by $Q_2$.  Since $\qbr{Q_2}$ is a nonzero multiple of $E_2$, exactness gives
\[
  \qbr{A_p(Q_2)}=T_{p,2}\qbr{Q_2}=(1+p)\qbr{Q_2}.
\]
By the injectivity of $L_2\to\QM_2$ from Lemma~\ref{lem:low-weight-inj}, this forces
\[
  A_p(Q_2)=(1+p)Q_2.
\]

If $A_p$ were multiplicative, then its value on $Q_2^2$ would be determined by its value on $Q_2$:
\[
  A_p(Q_2^2)=A_p(Q_2)^2=(1+p)^2Q_2^2.
\]
Taking q-brackets gives
\[
  \qbr{A_p(Q_2^2)}=(1+p)^2\qbr{Q_2^2}.
\]
In particular, the coefficient of the top-depth term $E_2^2$ would be multiplied by $(1+p)^2$.

However, exactness in weight $4$ says instead that
\[
  \qbr{A_p(Q_2^2)}=T_{p,4}\qbr{Q_2^2}.
\]
By Lemma~\ref{lem:top-hecke}, the top $E_2^2$ coefficient of a weight-$4$ quasimodular form is multiplied by
\[
  p^{2-1}\sigma_1(p)=p(1+p).
\]
The top $E_2^2$ coefficient of $\qbr{Q_2^2}$ is nonzero because $\mu(Q_2^2)=1/4$.  Thus the same nonzero coefficient would have to be multiplied both by $(1+p)^2$ and by $p(1+p)$.  Since
\[
  (1+p)^2-p(1+p)=1+p\ne0,
\]
this is impossible.  Hence no multiplicative exact lift exists.
\end{proof}

\begin{remark}
\label{rmk:mult-obstruction}
Theorem~\ref{thm:no-multiplicative} is an expected obstruction, since Hecke operators are not multiplicative and therefore no lift of them can be.  Its content is quantitative rather than qualitative.  The proof isolates the exact discrepancy, namely that the top-depth multiplier is $p\sigma_1(p)$ while multiplicativity would force $\sigma_1(p)^2$, and Theorem~\ref{thm:no-q2-semi} runs the same kind of single-coefficient comparison on a different invariant, with the constant term at the empty partition in place of the top-depth coefficient.  The obstruction is not an artifact of the choice of generators $Q_j$.  It is a statement about multiplication in $\Lam$ against the Hecke action on $\QM$.  Any exact lift is therefore linear and correction-term driven rather than a ring endomorphism.
\end{remark}

The next theorem rules out a weaker naturality condition.  Multiplication by $Q_2$ is the most basic operation that raises quasimodular depth.  An exact lift might be expected to commute with this operation up to the weight-$2$ Hecke eigenvalue.  This is already impossible on the Eisenstein line.

\begin{definition}[Strict $Q_2$-semilinearity]
Fix a prime $p$.  A family of exact lifts is strictly $Q_2$-semilinear on an element $f\in L_{2k}$ if
\[
  A_{p,2k+2}(Q_2f)=(1+p)Q_2A_{p,2k}(f).
\]
\end{definition}

\begin{theorem}[No strict $Q_2$-semilinear Eisenstein lift]
\label{thm:no-q2-semi}
For every prime $p$, no exact Hecke lift on $\Lam$ is strictly $Q_2$-semilinear on the element $e_4=Q_2^2+2Q_4$.
\end{theorem}

\begin{proof}
By Lemma~\ref{lem:low-weight-inj}, $\qbr{e_4}$ is a nonzero multiple of $E_4$.  Exactness in weight $4$ therefore determines the value of $A_{p,4}$ on the line spanned by $e_4$.  Since
\[
  T_{p,4}E_4=(1+p^3)E_4,
\]
we must have
\[
  A_{p,4}(e_4)=(1+p^3)e_4.
\]

Suppose strict $Q_2$-semilinearity held on $e_4$.  Then
\[
\begin{aligned}
  A_{p,6}(Q_2e_4)
  &=(1+p)Q_2A_{p,4}(e_4)\\
  &=(1+p)(1+p^3)Q_2e_4.
\end{aligned}
\]
Taking q-brackets, this would imply that the constant term of $\qbr{Q_2e_4}$ is multiplied by $(1+p)(1+p^3)$.

We now compute that constant term and compare it with exactness.  The constant term of a normalized q-bracket is the value of the function on the empty partition: the numerator has constant term $f(\varnothing)$ and the denominator has constant term $1$.  Thus
\[
  \operatorname{CT}\qbr{Q_2e_4}=(Q_2e_4)(\varnothing).
\]
Using
\[
  Q_2(\varnothing)=-\frac1{24},
  \qquad
  e_4(\varnothing)=\frac1{240},
\]
we get
\[
  (Q_2e_4)(\varnothing)= -\frac1{24}\cdot\frac1{240}=-\frac1{5760}\ne0.
\]

Since $Q_2e_4$ has weight $6$, exactness gives
\[
  \qbr{A_{p,6}(Q_2e_4)}=T_{p,6}\qbr{Q_2e_4}.
\]
Under the Hecke operator $T_{p,6}$, the constant term of a weight-$6$ form is multiplied by
\[
  \sigma_5(p)=1+p^5.
\]
Therefore exactness requires the nonzero constant term above to be multiplied by $1+p^5$.  Strict $Q_2$-semilinearity would require the same nonzero number to be multiplied by
\[
  (1+p)(1+p^3)=1+p+p^3+p^4.
\]
These two multipliers are unequal for every prime $p$.  Indeed,
\[
  1+p^5-(1+p+p^3+p^4)=p(p^4-p^3-p^2-1),
\]
and $p^4-p^3-p^2-1>0$ for every prime $p\ge2$, since $p^4\ge 2p^3=p^3+p^3>p^3+p^2+1$.  This contradiction proves the theorem.
\end{proof}
Therefore no family of exact lifts on $\Lam$ is fully $Q_2$-semilinear in every weight, since the pair of weights $(4,6)$ already fails. We restate that this also obstructs multiplicativity. Whether full $Q_2$-semilinearity fails at a given weight $2k>4$ is a separate question.\section{Injectivity of the lowering operator and kernel transport}
\label{sec:Binj}

This section proves \cref{thm:main-injectivity}.  We first prove that $B$ is injective on the genuine subspace, and then use this to transport Hecke actions on kernels between adjacent weights. Recall from \cref{sec:intro} the definitions of $L_m$ and $R_m$.

\begin{theorem}[Injectivity of $B$ on the genuine subspace]
\label{thm:Binj}
For every $m\ge4$, the restriction
\[
  B:L_m\to R_{m-2}
\]
is injective.
\end{theorem}

\begin{proof}
Filter $R$ by monomial length: a monomial $Q_{\lambda_1}\cdots Q_{\lambda_r}$ has length $r$.  We compare the highest-length part of $Bf$ with the highest-length part of $f$.  The operator $D$ is second order and replaces two factors by one factor, so it lowers length by one.  Thus $D$ contributes nothing to the associated graded for the length filtration.

The operator $\partial$ sends $Q_j$ to $Q_{j-1}$.  When $j>1$, this preserves the number of factors.  When $j=1$, it sends $Q_1$ to $Q_0=1$, which lowers length.  On the associated graded, therefore, the length-preserving part of $\partial$ is the derivation
\[
  \delta(Q_j)=Q_{j-1}\quad(j\ge2),
  \qquad
  \delta(Q_1)=0.
\]
Consequently the associated-graded operator of $B=\frac12(D-\partial^2)$ is
\[
  \operatorname{gr}(B)=-\frac12\delta^2.
\]
To prove injectivity of $B$ it is enough to prove injectivity of $\delta^2$ on each homogeneous length component of $L_m$.  Indeed, if $0\ne f\in L_m$ and $f_r$ is its highest nonzero length component, then all lower length components of $f$ can contribute only terms of length at most $r-1$ to $Bf$, while the length-$r$ component of $Bf$ is $-\frac12\delta^2 f_r$.  Thus this top component cannot be cancelled by lower-length terms.  We therefore prove that $\delta^2$ is injective on every length component of $L_m$.

Fix a length $r$.  The length-$r$ part of $L_m$ is spanned by monomials $Q_{\lambda_1}\cdots Q_{\lambda_r}$ with all $\lambda_i\ge2$ and $\sum_i\lambda_i=m$.  Identify this space with symmetric homogeneous polynomials by sending each factor $Q_j$ to $x^{j-1}/(j-1)!$ and symmetrizing in $r$ variables.  Under this identification, applying $\delta$ to one factor corresponds to differentiating with respect to the corresponding variable, so $\delta$ becomes
\[
  \Delta=\sum_{i=1}^r\frac{\partial}{\partial x_i}.
\]
Because all parts $\lambda_i$ are at least $2$, every monomial in the corresponding polynomial contains each variable at least once.  Hence the associated symmetric polynomial $F(x_1,\ldots,x_r)$ is divisible by
\[
  x_1x_2\cdots x_r.
\]

Suppose now that $\Delta^2F=0$.  For arbitrary $x=(x_1,\ldots,x_r)$, define a one-variable polynomial
\[
  G_x(t)=F(x_1+t,\ldots,x_r+t).
\]
By the chain rule,
\[
  G_x'(t)=(\Delta F)(x_1+t,\ldots,x_r+t),
  \qquad
  G_x''(t)=(\Delta^2F)(x_1+t,\ldots,x_r+t)=0.
\]
Thus $G_x(t)$ is affine in $t$.  But $F$ is divisible by $x_1\cdots x_r$, so $G_x(t)$ is divisible by
\[
  \prod_{i=1}^r(x_i+t).
\]
If $r\ge2$ and the $x_i$ are chosen pairwise distinct, this product has at least two distinct roots as a polynomial in $t$.  A nonzero affine polynomial cannot have two distinct roots, so $G_x(t)$ must be identically zero for all such $x$.  Since the set of $x$ with pairwise distinct coordinates is Zariski dense, this forces $F=0$.

When $r=1$, the length-one component of $L_m$ is spanned by $Q_m$, which corresponds to the polynomial $x^{m-1}/(m-1)!$.  Since $m\ge4$, its second derivative is nonzero.  Thus $\delta^2$ is injective on every length component of $L_m$.  The associated-graded argument above now shows that $Bf\ne0$ for every nonzero $f\in L_m$.
\end{proof}

\begin{corollary}
\label{cor:Binj-kernel}
For $k\ge2$, the map
\[
  B:K^\Lambda_{2k}\to K_{2k-2}
\]
is injective.
\end{corollary}

\begin{proof}
If $u\in K_{2k}$, then $\qbr{Bu}=d\qbr{u}=0$ by \eqref{eq:zagier-lower}.  Thus $B$ maps $K^\Lambda_{2k}$ into $K_{2k-2}$.  Injectivity follows from Theorem~\ref{thm:Binj}.
\end{proof}

\begin{definition}[Kernel-strict lowering]
\label{def:kernel-strict}
A family of exact lifts preserves the genuine kernel in weight $2k$ if
\[
  A_{n,2k}(K^\Lambda_{2k})\subseteq K^\Lambda_{2k}
  \qquad(n\ge1).
\]
It satisfies kernel-strict lowering in weight $2k$ if
\[
  BA_{n,2k}u=nA_{n,2k-2}Bu
  \qquad(u\in K^\Lambda_{2k},\ n\ge1).
\]
\end{definition}

Hecke lifts satisfying both conditions exist.  For example, the scalar-kernel lifts constructed in \cref{sec:scalar} preserve the genuine kernels and satisfy kernel-strict lowering (\cref{cor:scalar-strict}).

\begin{theorem}[Kernel transport]
\label{thm:kernel-transport}
Let $\{A_{n,2j}\}$ be exact Hecke lifts preserving the genuine kernels and satisfying kernel-strict lowering.  Set
\[
  W_{2k-2}=B(K^\Lambda_{2k})\subseteq K_{2k-2}.
\]
Then $W_{2k-2}$ is stable under $A_{n,2k-2}$ for every $n$, and
\[
  B:K^\Lambda_{2k}\xrightarrow{\sim}W_{2k-2}
\]
intertwines the Hecke action on $K^\Lambda_{2k}$ with the $n$-twisted lower-weight action on $W_{2k-2}$.  In particular, if $f\in K^\Lambda_{2k}$ satisfies $A_{n,2k}f=\alpha_nf$ for all $n$, then $Bf\neq0$ and $A_{n,2k-2}(Bf)=(\alpha_n/n)\,Bf$.
\end{theorem}

\begin{proof}
Let $w=Bu\in W_{2k-2}$ with $u\in K^\Lambda_{2k}$.  Kernel-strict lowering gives
\[
  nA_{n,2k-2}w=nA_{n,2k-2}Bu=BA_{n,2k}u.
\]
Since the lift preserves the genuine kernel, $A_{n,2k}u\in K^\Lambda_{2k}$, and hence the right side lies in $B(K^\Lambda_{2k})=W_{2k-2}$.  Since $n\neq0$, $A_{n,2k-2}w\in W_{2k-2}$.  The same identity gives the intertwining relation, and Corollary~\ref{cor:Binj-kernel} gives the isomorphism.  For the last assertion, apply $B$ to $A_{n,2k}f=\alpha_nf$ and use kernel-strict lowering: $nA_{n,2k-2}Bf=BA_{n,2k}f=\alpha_nBf$, and $Bf\neq0$ by Theorem~\ref{thm:Binj}.
\end{proof}

Assume the hypotheses of Theorem~\ref{thm:kernel-transport}.  Put $r_k=\dim K^\Lambda_{2k}$ and, for each $n\ge1$, set
\[
  \chi^\Lambda_{n,k}(X)=\det\left(XI-A_{n,2k}|_{K^\Lambda_{2k}}\right),
  \qquad
  \chi^K_{n,k-1}(X)=\det\left(XI-A_{n,2k-2}|_{K_{2k-2}}\right).
\]

\begin{proposition}[Divisibility of characteristic polynomials]
\label{thm:spectral-divisibility}
Under these hypotheses,
\[
  n^{-r_k}\chi^\Lambda_{n,k}(nX)
\]
divides $\chi^K_{n,k-1}(X)$ in $\QQ[X]$.
\end{proposition}

\begin{proof}
By Theorem~\ref{thm:kernel-transport}, the map
\[
  B:K^\Lambda_{2k}\longrightarrow W_{2k-2}
\]
is an isomorphism and satisfies the intertwining identity
\[
  B A_{n,2k}=n A_{n,2k-2}B
  \qquad\text{on }K^\Lambda_{2k}.
\]
Thus, after transporting the action through $B$, the operator $A_{n,2k}$ on $K^\Lambda_{2k}$ is conjugate to the operator $nA_{n,2k-2}$ on $W_{2k-2}$.  Therefore
\[
  \chi^\Lambda_{n,k}(X)
  =\det\left(XI-nA_{n,2k-2}|_{W_{2k-2}}\right).
\]
Let
\[
  \rho_{n,k-1}(X)=\det\left(XI-A_{n,2k-2}|_{W_{2k-2}}\right).
\]
If $r_k=\dim W_{2k-2}=\dim K^\Lambda_{2k}$, then
\[
\begin{aligned}
  n^{-r_k}\chi^\Lambda_{n,k}(nX)
  &=n^{-r_k}\det\left(nXI-nA_{n,2k-2}|_{W_{2k-2}}\right)\\
  &=n^{-r_k}n^{r_k}\det\left(XI-A_{n,2k-2}|_{W_{2k-2}}\right)\\
  &=\rho_{n,k-1}(X).
\end{aligned}
\]
Finally, $W_{2k-2}$ is invariant under $A_{n,2k-2}$ inside the full kernel $K_{2k-2}$.  Choosing a basis of $K_{2k-2}$ beginning with a basis of $W_{2k-2}$ makes the matrix of $A_{n,2k-2}|_{K_{2k-2}}$ block upper triangular.  Its characteristic polynomial is therefore the product of $\rho_{n,k-1}(X)$ with the characteristic polynomial of the induced operator on the quotient $K_{2k-2}/W_{2k-2}$.  Hence $\rho_{n,k-1}(X)$, equivalently $n^{-r_k}\chi^\Lambda_{n,k}(nX)$, divides $\chi^K_{n,k-1}(X)$.
\end{proof}

The divisibility has $p$-adic consequences; compare the role of $p$-adic $q$-brackets in \cite{GriffinJamesonTrebatLeder}.

\begin{corollary}[$p$-adic constraints on kernel eigenvalues]
\label{cor:padic}
Let $p$ be prime, and suppose the lower-weight action on $W_{2k-2}$ preserves a $\ZZ_p$-lattice.  Write
\[
  \chi^\Lambda_{p,k}(X)=X^{r_k}+a_1X^{r_k-1}+\cdots+a_{r_k}.
\]
Then $a_j\in p^j\ZZ_p$ for $1\le j\le r_k$, and every eigenvalue $\alpha_p$ of $A_{p,2k}$ on $K^\Lambda_{2k}$ satisfies $v_p(\alpha_p)\ge1$; that is, no eigenvector in $K^\Lambda_{2k}$ is ordinary at $p$.  In particular, if $K^\Lambda_{2k}\neq0$, the genuine kernel action is never the Eisenstein scalar action $A_{p,2k}|_{K^\Lambda_{2k}}=(1+p^{2k-1})\Id$.
\end{corollary}

\begin{proof}
By Theorem~\ref{thm:kernel-transport}, the eigenvalues of $A_{p,2k}$ on $K^\Lambda_{2k}$ are $p\beta_1,\ldots,p\beta_{r_k}$, where the $\beta_i$ are eigenvalues of the lower-weight action on $W_{2k-2}$.  Since that action preserves a $\ZZ_p$-lattice, the $\beta_i$ are integral over $\ZZ_p$.  Therefore $v_p(p\beta_i)\ge1$ for each $i$, and the coefficient $a_j$, which is up to sign the $j$th elementary symmetric polynomial in the $p\beta_i$, lies in $p^j\ZZ_p$.  The Eisenstein scalar $1+p^{2k-1}$ is a $p$-adic unit, so it cannot occur as an eigenvalue.
\end{proof}
\begin{remark}
\label{rmk:transport-significance}
\cref{thm:spectral-divisibility} and \cref{cor:padic} are the first place where injectivity of $B$ yields arithmetic rather than linear-algebraic information.  Injectivity alone says that $K^\Lambda_{2k}$ embeds in $K_{2k-2}$.  Kernel-strict lowering upgrades that embedding to an intertwiner, and the twist by $n$ in \eqref{eq:dT-twist} is then carried into the spectrum, so that an eigenvalue of $A_{n,2k}$ on the genuine kernel is $n$ times an eigenvalue of $A_{n,2k-2}$ below.  Thus the $q$-bracket kernels in adjacent weights are not spectrally independent, and every divisibility statement above descends from that single factor of $n$.  However, the only lifts constructed here that satisfy kernel-strict lowering are the scalar-kernel lifts of \cref{sec:scalar}, whose genuine kernel action is a scalar, so these results should be read as constraints on the hypothetical natural lifts of \cref{q:intrinsic} rather than as statements about lifts already in hand.
\end{remark}
\begin{proposition}[Transported subspaces in large weights]
\label{thm:infinite-subspaces}
Under the hypotheses of Theorem~\ref{thm:kernel-transport}, the lower formal kernels $K_{2k-2}$ contain nonzero proper Hecke-stable subspaces
\[
  W_{2k-2}=B(K^\Lambda_{2k})
\]
for all sufficiently large $k$.
\end{proposition}

\begin{proof}
We have
\[
  \dim L_{2k}=p(2k)-p(2k-1),
\]
where $p(m)$ is the partition function, while
\[
  \dim \QM_{2k}=\#\{(a,b,c)\in\ZZ_{\ge0}^3:a+2b+3c=k\}=O(k^2).
\]
By the Hardy--Ramanujan asymptotic \cite{HardyRamanujan}, $p(2k)-p(2k-1)$ grows exponentially in $\sqrt{k}$ divided by a polynomial factor, so $\dim K^\Lambda_{2k}>0$ for all sufficiently large $k$.  Injectivity of $B$ then gives $W_{2k-2}\neq0$.

For properness, the formal kernel $K_{2k-2}$ contains the entire subspace $Q_1R_{2k-3}$, so
\[
  \dim K_{2k-2}\ge p(2k-3).
\]
On the other hand,
\[
  \dim W_{2k-2}=\dim K^\Lambda_{2k}\le \dim L_{2k}=p(2k)-p(2k-1).
\]
Again by Hardy--Ramanujan \cite{HardyRamanujan},
\[
  \frac{p(2k)-p(2k-1)}{p(2k-3)}\to0.
\]
Thus $\dim W_{2k-2}<\dim K_{2k-2}$ for all sufficiently large $k$.  Stability was proved in Theorem~\ref{thm:kernel-transport}.
\end{proof}

\section{Scalar-kernel exact lifts}
\label{sec:scalar}

This section constructs the scalar-kernel lifts of \cref{cor:main-16}.  Recall that $\sigma_r(n)=\sum_{d\mid n}d^r$.  The divisor-sum identity
\begin{equation}
  \sigma_r(m)\sigma_r(n)=\sum_{d\mid(m,n)}d^r\sigma_r(mn/d^2)
  \label{eq:sigma-id}
\end{equation}
is called the Eisenstein Hecke relation.

\begin{lemma}[Scalar-kernel identities]
\label{lem:lambda-identities}
For $k\ge2$, set
\[
  \lambda_{n,2k}=n^{k-2}\sigma_3(n).
\]
Then
\[
  \lambda_{m,2k}\lambda_{n,2k}
  =\sum_{d\mid(m,n)}d^{2k-1}\lambda_{mn/d^2,2k},
\]
and for $k\ge3$,
\[
  \lambda_{n,2k}=n\lambda_{n,2k-2}.
\]
\end{lemma}

\begin{proof}
Using \eqref{eq:sigma-id} with $r=3$,
\[
  \lambda_{m,2k}\lambda_{n,2k}
  =(mn)^{k-2}\sum_{d\mid(m,n)}d^3\sigma_3(mn/d^2).
\]
On the other hand,
\[
  \sum_{d\mid(m,n)}d^{2k-1}\lambda_{mn/d^2,2k}
  =\sum_{d\mid(m,n)}d^{2k-1}\left(\frac{mn}{d^2}\right)^{k-2}\sigma_3(mn/d^2),
\]
and the power of $d$ is $2k-1-2(k-2)=3$.  The adjacent identity is immediate from
\[
  n^{k-2}\sigma_3(n)=n\cdot n^{k-3}\sigma_3(n).
\]
\end{proof}

\begin{definition}[Scalar-kernel lift]
\label{def:transported-scalar}
Let $2k\ge4$ and assume $I_{2k}$ is Hecke-stable.  Choose a section
\[
  s_{2k}:I_{2k}\to R_{2k}
\]
of $q_{2k}$.  The scalar-kernel operator is the map
\[
  A^{\mathrm{sc}}_{n,2k}\bigl(s_{2k}(\phi)+u\bigr)
  =s_{2k}(T_{n,2k}\phi)+\lambda_{n,2k}u,
  \qquad
  \phi\in I_{2k},\ u\in K_{2k}.
\]
\end{definition}

\begin{theorem}[Scalar-kernel exact lift]
\label{thm:scalar-lift}
The operator $A^{\mathrm{sc}}_{n,2k}$ is an exact lift of $T_{n,2k}$, and the family $\{A^{\mathrm{sc}}_{n,2k}\}_{n\ge1}$ satisfies the Hecke algebra relations in weight $2k$.
\end{theorem}

\begin{proof}
Every element of $R_{2k}$ is written uniquely as $s_{2k}(\phi)+u$ with $\phi\in I_{2k}$ and $u\in K_{2k}$.  First check exactness.  Since $s_{2k}$ is a section and $u$ lies in the kernel,
\[
\begin{aligned}
  q_{2k}A^{\mathrm{sc}}_{n,2k}\bigl(s_{2k}(\phi)+u\bigr)
  &=q_{2k}\bigl(s_{2k}(T_{n,2k}\phi)+\lambda_{n,2k}u\bigr)\\
  &=T_{n,2k}\phi+\lambda_{n,2k}q_{2k}(u)\\
  &=T_{n,2k}\phi.
\end{aligned}
\]
But $q_{2k}(s_{2k}(\phi)+u)=\phi$, so this is precisely
\[
  q_{2k}A^{\mathrm{sc}}_{n,2k}=T_{n,2k}q_{2k}.
\]

It remains to check the Hecke algebra relations.  On the quotient part $s_{2k}(I_{2k})$, the operator is transported directly from the Hecke action on $I_{2k}$, so the quotient blocks satisfy the usual relation
\[
  T_{m,2k}T_{n,2k}=\sum_{d\mid(m,n)}d^{2k-1}T_{mn/d^2,2k}.
\]
On the kernel $K_{2k}$, the operator $A^{\mathrm{sc}}_{n,2k}$ is multiplication by the scalar $\lambda_{n,2k}$.  Therefore the kernel part of the Hecke relation is exactly the scalar identity
\[
  \lambda_{m,2k}\lambda_{n,2k}
  =\sum_{d\mid(m,n)}d^{2k-1}\lambda_{mn/d^2,2k},
\]
which is Lemma~\ref{lem:lambda-identities}.  There is no off-diagonal correction term in this construction.  Thus both the quotient and kernel blocks satisfy the required Hecke relations, and the family $\{A^{\mathrm{sc}}_{n,2k}\}_{n\ge1}$ is a full exact Hecke lift.
\end{proof}

\begin{corollary}[Strict lowering for scalar-kernel lifts]
\label{cor:scalar-strict}
Assume the scalar-kernel construction is made in adjacent weights $2k$ and $2k-2$, with $k\ge3$.  Then for every $u\in K_{2k}$,
\[
  BA^{\mathrm{sc}}_{n,2k}u=nA^{\mathrm{sc}}_{n,2k-2}Bu.
\]
In particular the scalar-kernel lifts preserve $K^\Lambda_{2k}$ and satisfy kernel-strict lowering on $K^\Lambda_{2k}$.
\end{corollary}

\begin{proof}
For $u\in K_{2k}$,
\[
  BA^{\mathrm{sc}}_{n,2k}u=\lambda_{n,2k}Bu.
\]
Since $Bu\in K_{2k-2}$,
\[
  nA^{\mathrm{sc}}_{n,2k-2}Bu=n\lambda_{n,2k-2}Bu.
\]
The two expressions agree by Lemma~\ref{lem:lambda-identities}.  If $u\in K^\Lambda_{2k}$, then $A^{\mathrm{sc}}_{n,2k}u=\lambda_{n,2k}u$, so the genuine kernel is preserved.
\end{proof}

\begin{remark}
The scalar-kernel construction is noncanonical, since it depends on the section $s_{2k}$.  It supplies explicit exact lifts satisfying the strict kernel-lowering relation used in the transport theorem.
\end{remark}

\section{Finite rank evidence through weight 16}
\label{sec:computations}

This short section records the finite computational output used in the main text.  The details of the exact-rational matrix construction and the low-weight kernel vectors are placed in \cref{app:computational}.

\begin{proposition}[Finite rank certificate and genuine kernel dimensions through weight 16]
\label{prop:kerneldims}
For every even weight $2k\le16$, the exact rational computation gives
\[
  \qbr{L_{2k}}=\QM_{2k}.
\]
The dimensions are
\[
\begin{array}{c|rrrrrrr}
2k&4&6&8&10&12&14&16\\ \hline
\dim L_{2k}&2&4&7&12&21&34&55\\
\dim \QM_{2k}&2&3&4&5&7&8&10\\
\dim K^\Lambda_{2k}&0&1&3&7&14&26&45
\end{array}
\]
\end{proposition}

\begin{proof}
The computation is carried out over $\QQ$ using the monomial basis of $L_{2k}$ and the coefficient matrix of the q-brackets through $q^{28}$.  In the seven displayed weights, exact row reduction gives ranks
\[
  2,3,4,5,7,8,10,
\]
which agree with the dimensions of $\QM_{2k}$.  Since the Bloch--Okounkov--Zagier theorem gives $\qbr{L_{2k}}\subseteq\QM_{2k}$, the full q-bracket image has precisely this rank.  The kernel dimensions are therefore $\dim L_{2k}-\dim\QM_{2k}$.  The construction of the matrices and the exact rank certificate are given in \cref{app:computational}.
\end{proof}

\begin{corollary}[Exact scalar-kernel lifts through weight $16$]
\label{cor:through16}
For every even weight
\[
  2k=4,6,8,10,12,14,16,
\]
the scalar-kernel construction gives a full exact Hecke lift of the complete level-one quasimodular Hecke action on $R_{2k}$.  Between adjacent computed weights it satisfies kernel-strict lowering.
\end{corollary}

\begin{proof}
By \cref{prop:kerneldims}, the q-bracket image is $\QM_{2k}$ in these weights, and $\QM_{2k}$ is Hecke-stable.  Apply \cref{thm:scalar-lift}.  The strict-lowering statement follows from \cref{cor:scalar-strict}.
\end{proof}

\begin{remark}
The finite rank certificate verifies the low-weight examples and identifies the weights in which the scalar-kernel construction gives a lift on the full quasimodular target.  It is a finite-weight computation, not an all-weight surjectivity theorem.
\end{remark}

\section{Discussion and open problems}
\label{sec:discussion}

The results above separate three related questions.  First, exact lifts exist formally whenever the q-bracket image is Hecke-stable, or surjective,  although this construction depends on a section of $q_{2k}:R_{2k}\to I_{2k}$.  Second, multiplicative lifts and strictly $Q_2$-semilinear lifts do not exist; this is consistent with the fact that Hecke operators on quasimodular forms are themselves nonmultiplicative.  Third, the scalar-kernel construction gives exact formal lifts satisfying strict kernel lowering assuming the q-bracket's image is stable, and the injectivity of $B$ converts that condition into eigenvalue and characteristic polynomial constraints on q-bracket kernels.

Several natural problems remain.

\begin{question}[All-weight surjectivity]
\label{q:surjectivity}
Is
\[
  \qbr{L_{2k}}=\QM_{2k}
\]
for every $k$?  The computations in this paper verify this through weight $16$.
\end{question}
We believe that the q-bracket is indeed surjective.  The dimension count is lopsided, since $\dim L_{2k}$ grows faster than any power of $k$ while $\dim\QM_{2k}$, the number of partitions of $k$ into at most three parts, grows quadratically.  However, a large source does not by itself force surjectivity, and Corollary~\ref{cor:harmonic-reduction} reduces the question to harmonic surjectivity together with depth triangularity of multiplication by $Q_2$.
\begin{question}[Intrinsic nonmultiplicative lifts]
\label{q:intrinsic}
Does there exist a canonical exact Hecke lift on $\Lam$ that is linear but nonmultiplicative, with explicit correction terms accounting for the defects in Theorems~\ref{thm:no-multiplicative} and~\ref{thm:no-q2-semi}?
\end{question}

\begin{question}[Classification of correction defects]
Can the $Q_2$-defect
\[
  \Delta_n(f)=A_{n,2k+2}(Q_2f)-\sigma_1(n)Q_2A_{n,2k}(f)
\]
be described intrinsically in terms of Zagier's differential operators or q-bracket kernel projections?
\end{question}

\begin{question}[Non-scalar kernel modules]
Classify Hecke module structures on $K_{2k}$ satisfying the cocycle identities and kernel-strict lowering.  The scalar-kernel family gives one baseline model, but the obstruction theorems suggest that more natural lifts, if they exist, must have nontrivial kernel-valued correction terms.
\end{question}

\appendix

\section{Exact computations through weight 16}
\label{app:computational}

This appendix records the finite exact-rational data behind \cref{prop:kerneldims}.  All ranks are computed over $\QQ$ from q-bracket coefficient matrices; no numerical approximation is used.

\subsection{Rank certificate}
For each even weight $2k\le16$, let $\mathcal B_{2k}$ be the monomial basis of $L_{2k}$ indexed by partitions of $2k$ with all parts at least $2$.  Let $M_{2k}^{(28)}$ be the matrix whose columns are the coefficients from $q^0$ through $q^{28}$ of the q-brackets $\qbr{Q_\alpha}$ for $Q_\alpha\in\mathcal B_{2k}$.  Exact row reduction gives
\[
\begin{array}{c|rrrrrrr}
2k&4&6&8&10&12&14&16\\ \hline
\dim L_{2k}&2&4&7&12&21&34&55\\
\rank_{\QQ}M_{2k}^{(28)}&2&3&4&5&7&8&10\\
\dim\QM_{2k}&2&3&4&5&7&8&10\\
\dim K^\Lambda_{2k}&0&1&3&7&14&26&45
\end{array}
\]
Since the Bloch--Okounkov--Zagier theorem gives $\qbr{L_{2k}}\subseteq\QM_{2k}$, equality of the computed rank with $\dim\QM_{2k}$ proves $\qbr{L_{2k}}=\QM_{2k}$ in the displayed weights.  This is the finite rank certificate used in \cref{prop:kerneldims}.

\subsection{Low-weight kernel data}
Van Ittersum's modularity criterion gives a natural harmonic description of modular q-brackets in the genuine shifted-symmetric space~\cite{vanIttersumModular}.  The nullspace vectors below are written in a computational monomial basis rather than in that harmonic basis.  In weight $6$, the vector $g_6$ lies in the span of the harmonic modular representatives associated in that notation to the partitions $(6)$ and $(3,3)$.  In weight $8$, the three-dimensional genuine kernel can similarly be reorganized using $g_6Q_2$ together with suitable combinations of the harmonic representatives associated to $(8)$, $(5,3)$, and $(4,4)$.  We keep the monomial-basis vectors because they make the exact rank certificate and the action of $B$ completely explicit.

The first genuine kernel occurs in weight $6$:
\[
  g_6=161Q_6-109Q_4Q_2+77Q_3^2+5Q_2^3,
  \qquad
  B(g_6)=32Q_3Q_1-15Q_2Q_1^2\neq0.
\]
In weight $8$, with ordered basis
\[
(Q_8,Q_6Q_2,Q_5Q_3,Q_4^2,Q_4Q_2^2,Q_3^2Q_2,Q_2^4),
\]
the exact nullspace has basis
\begin{align*}
 u_{8,1}={}&3382Q_8-1805Q_6Q_2+1133Q_5Q_3-297Q_4^2+35Q_4Q_2^2,\\
 u_{8,2}={}&3144Q_8-2680Q_6Q_2+111Q_5Q_3+676Q_4^2+35Q_3^2Q_2,\\
 u_{8,3}={}&5062Q_8+610Q_6Q_2+4598Q_5Q_3-3377Q_4^2+7Q_2^4.
\end{align*}
Applying $B=\frac12(D-\partial^2)$ gives
\begin{align*}
 B(u_{8,1})={}&-\frac72 g_6+\frac{2477}{2}Q_5Q_1-35Q_4Q_1^2-70Q_3Q_2Q_1,\\
 B(u_{8,2})={}&-7g_6+\frac{5249}{2}Q_5Q_1-105Q_3Q_2Q_1,\\
 B(u_{8,3})={}&14g_6-2909Q_5Q_1-42Q_2^2Q_1^2.
\end{align*}
These identities are obtained by exact row reduction and direct application of $B$ to the displayed monomial bases.  The vectors $g_6,u_{8,1},u_{8,2},u_{8,3}$ are included only as low-weight certificates and examples; the higher-dimensional kernel data used in the paper is summarized by the rank table above.

\begin{corollary}[Transported spaces through weight $16$]
For the scalar-kernel lift, the subspaces
\[
  B(K^\Lambda_{12})\subset K_{10},
  \qquad
  B(K^\Lambda_{14})\subset K_{12},
  \qquad
  B(K^\Lambda_{16})\subset K_{14}
\]
are Hecke-stable of dimensions $14,26,45$, respectively.
\end{corollary}

\begin{proof}
The dimensions are those in \cref{prop:kerneldims}.  Injectivity of $B$ on genuine kernels follows from \cref{cor:Binj-kernel}, and Hecke-stability follows from \cref{thm:kernel-transport} and \cref{cor:scalar-strict}.
\end{proof}

\section{Conditional harmonic reduction}
\label{app:harmonic}

This appendix records a conditional reduction, motivated by the harmonic approach to modular q-brackets in \cite{vanIttersumModular}.  Full shifted-symmetric surjectivity would follow from harmonic modular surjectivity together with depth triangularity of multiplication by $Q_2$.

Let
\[
  q_k:L_k\to\QM_k,
  \qquad q_k(f)=\qbr{f}.
\]
Let $M_k\subseteq \QM_k$ denote the modular subspace.  Let $\mathcal H_k$ denote the harmonic subspace in the shifted-symmetric decomposition of van Ittersum.

We use two standard structural inputs.  First, the shifted-symmetric algebra has a harmonic decomposition
\[
  L_k=\bigoplus_{r\ge0}Q_2^r\mathcal H_{k-2r}.
\]
Second, modular q-brackets are represented by harmonic elements modulo the q-bracket kernel: if $f\in L_k$ and $\qbr{f}$ is modular, then
\[
  f=h+\kappa,
  \qquad h\in\mathcal H_k,
  \quad \kappa\in\ker q_k.
\]

\begin{definition}[Depth triangularity]
We say that the q-bracket satisfies depth triangularity in weight $k$ if, for every $r\ge0$ and $h\in\mathcal H_{k-2r}$,
\[
  \qbr{Q_2^rh}=\gamma_{k,r}E_2^r\qbr{h}+\text{terms of smaller }E_2\text{-depth}
\]
with $\gamma_{k,r}\in\QQ^\times$.
\end{definition}

\begin{proposition}[Full surjectivity implies harmonic surjectivity]
\label{prop:full-to-harmonic}
If $q_k(L_k)=\QM_k$, then
\[
  q_k(\mathcal H_k)=M_k.
\]
\end{proposition}

\begin{proof}
Let $F\in M_k$.  By full surjectivity choose $f\in L_k$ with $\qbr{f}=F$.  Since $F$ is modular, the modularity criterion gives $f=h+\kappa$ with $h\in\mathcal H_k$ and $\kappa\in\ker q_k$.  Therefore
\[
  F=\qbr{f}=\qbr{h}.
\]
\end{proof}

\begin{theorem}[Harmonic surjectivity implies full shifted-symmetric surjectivity]
\label{thm:harmonic-to-full}
Assume depth triangularity in weight $k$.  Suppose that
\[
  q_j(\mathcal H_j)=M_j
\]
for every $j\le k$ of the same parity as $k$.  Then
\[
  q_k(L_k)=\QM_k.
\]
\end{theorem}

\begin{proof}
Every quasimodular form of weight $k$ has a finite depth expansion
\[
  F=\sum_{r=0}^{\lfloor k/2\rfloor}E_2^rF_r,
  \qquad F_r\in M_{k-2r}.
\]
We prove that $F\in q_k(L_k)$ by descending induction on the largest $r$ for which $F_r\neq0$.  Let $R$ be that largest index.  By harmonic surjectivity in weight $k-2R$, choose $h_R\in\mathcal H_{k-2R}$ with
\[
  \qbr{h_R}=\gamma_{k,R}^{-1}F_R.
\]
Depth triangularity gives
\[
  \qbr{Q_2^Rh_R}=E_2^RF_R+\text{terms of smaller depth}.
\]
Subtracting this q-bracket from $F$ removes the top-depth term.  Repeating the process eventually expresses $F$ as a q-bracket of an element of $L_k$.
\end{proof}

\begin{corollary}[Reduction of all-weight surjectivity]
\label{cor:harmonic-reduction}
Assume depth triangularity in all weights.  Then
\[
  q_k(L_k)=\QM_k\quad\text{for all }k
\]
if and only if
\[
  q_k(\mathcal H_k)=M_k\quad\text{for all }k.
\]
\end{corollary}

\begin{proof}
One direction is Proposition~\ref{prop:full-to-harmonic}; the other is Theorem~\ref{thm:harmonic-to-full} weight by weight.
\end{proof}

\end{document}